\documentclass{article}
\usepackage{amsmath}
\usepackage{amssymb}
\usepackage{amsthm}
\usepackage{amsfonts}
\usepackage{diagbox}
\newtheorem{theorem}{Theorem}[section]
\newtheorem{proposition}[theorem]{Proposition}
\newtheorem{lemma}[theorem]{Lemma}
\newtheorem{corollary}[theorem]{Corollary}

\newtheorem{definition}[theorem]{Definition}

\newtheorem*{notation}{Notation}
\newtheorem*{remark}{Remark}
\newtheorem*{example}{Example}

\title{A universal linear algebraic model for conformal geometries}
\author{M\'at\'e L. Juh\'asz\\[2mm]
Alfr\'ed R\'enyi Institute of Mathematics\\
Hungarian Academy of Sciences\\
\texttt{juhasz.mate.lehel@renyi.mta.hu}\\
}

\DeclareMathOperator{\charop}{char}

\newcommand{\FSDet}{\cite{fillmore2000}}
\newcommand{\HLR}{\cite{hestens2000}}
\newcommand{\Kit}{\cite{kitaoka1993}}
\newcommand{\Mac}{\cite{macdonald2015}}
\newcommand{\Str}{\cite{struve2015}}
\newcommand{\Wil}{\cite{wildberger2009}}

\begin{document}

\maketitle

\begin{abstract}
This article describes an entirely algebraic construction for developing conformal geometries, which provide models for, among others, the Euclidean, spherical and hyperbolic geometries.
On one hand, their relationship is usually shown analytically, through a framework comparing the measurement of distances and angles in Cayley-Klein geometries, including Lorentzian geometries, as done by F. Bachmann and later R. Struve.
On the other hand, such a relationship may also be expressed in a purely linear algebraic manner, as explained by D. Hestens, H. Li and A. Rockwood.
\\
The model described in this article unifies these approaches via a generalization of Lie sphere geometry.
Like the work of N. Wildberger, it is a purely algebraic construction, and as such it works over
 any field of odd characteristic. It is shown that measurement of distances and angles is an inherent property of the model that is easy to identify, and the possible models are classified over the real, complex and finite fields, and partially in characteristic $2$, revealing a striking analogy between the real and finite geometries.
\end{abstract}

\section{Introduction}

The three classical plane geometries, spherical, Euclidean and hyperbolic, share deep connections between them. This connection can be formalized using linear algebra, as seen in \HLR{} and \Mac{}, where conformal geometric algebra is used.
Another way of approaching the similarities is using projective geometry and comparing how distances and angles are measured. As seen in \Str{}, we have three types of measurement for both distances and angles: spherical, Euclidean and hyperbolic. This gives us $9$ possible plane geometries. In the referenced article, a synthetic construction is given, based on the group of symmetries of the geometry, previously studied in \cite{bachmann1973}. This approach goes back as far as \cite{klein1926}, and it shows a natural duality between points and lines, something that is lost in the formulation found in \HLR{}. Similar methods are employed by Arkady Onishchik and Rolf Sulanke in \cite{onishchik2006}, in the development of Benz planes by Walter Benz (\cite{benz1973}), and the approach taken in the works of Emil Moln\'ar via polarity in projective geometry (\cite{molnar1976}).

These two approaches can be unified via a generalization of Lie's geometry of spheres (see \FSDet{}, \cite{fillmore1995}, \cite{fillmore1995planar}, \cite{lie1896}, \cite{rigby1981}), where lines and cycles are given orientation, and a quadratic hypersurface is constructed that contains all points, oriented lines and oriented cycles. This is a linear algebraic formulation, in the vein described in \HLR{}, but also more general, as the duality between points and lines is restored.

This article presents a formal framework that extends these previous works, and derives the measurment of distances and angles as a natural property of the group of isometries. This way distances and angles can be calculated for any field, without the need to use classical trigonometrical functions. The $9$ plane geometries in \Str{} are revealed in a natural way. Furthermore, similar enumerations may be done for all fields, presented here for the field of complex numbers and all finite fields of odd characteristic, resulting in a strikingly similar pattern of $9$ plane geometries for finite fields.

This construction can be developed without restricting the field of coordinates, and parts of the construction may even be used in the case when the characteristic is $2$. Among its properties, several of the classical models for these geometries, such as the Cayley-Klein and Poincar\'e models, may be derived directly, as seen in \HLR{} and \FSDet{}. The Poincar\'e model permits the construction of M\"obius planes, which over arbitrary fields become the Miquelian M\"obius planes (\cite{chen1970}).
Other attempts at describing conformal geometries over arbritrary fields include \Str{}, \Wil{} and \cite{wildberger2006}.
This article hopes to extend and incorporate all these constructions and shed some insight into these relationships through linear algebra.

\subsection{Overview}

The main definition that introduces the central object, the universal conformal geometry, is introduced in section 4. The main result of the article is announced in section 6, as a consequence of the theorem in subsection 5.3.

Section 1 and 2 give some overview of the main ideas behind the work. Before describing the universal formulation and how we can classify geometries, as a motivation I will describe in Section 2 the conformal unification of the Euclidean geometry with the curved geometries, and the process of ``quadratic lift''. Section 3 establishes notations and conventions in linear algebra.

Section 4 introduces the concept of universal conformal geometry, and gives several classical examples, as well as a few new ones. Section 5 presents some simple incidence properties, including a generalized form of the Poincar\'e and Cayley-Klein models, after which the measurement of distance is explained. Finally, section 6 classifies plane Cayley-Klein geometries in this general framework over the real numbers and finite fields of odd characteristic.

\section{Motivation}

An important distinction that is often made between homogeneous conformal geometries is whether they are modelled on an affine space, as is done with Euclidean geometry, or on a quadratic hypersurface inside a vector space (or projective space), as in the case of spherical or hyperbolic geometries. As it will turn out, this is an artifact of the construction rather than an inherent property of the geometry. I will rather describe this distinction in terms borrowed from differential geometry, referring to the former as a {\it flat geometry}, the later as a {\it curved geometry}.

Let us fix the field to be that of the real numbers. Consider these two definitions:

\begin{definition}
Let us call a {\bf curved geometry} of dimension $n$ a triple $(V,Q,r)$ consisting of a vector space $V$ of dimension $n+1$ over $\mathbb{R}$, a non-degenerate quadratic form $Q$ and $r$ a number. The points of the geometry are vectors $v\in V$ where $Q(v)=r$. The {\bf hyperplanes} of the geometry are linear subspaces of dimension $n$ of $V$, while {\bf hypercycles} are affine subspaces of dimension $n$. It is called a {\bf spherical space} when $Q$ is of signature $(n+1,0)$ and $r>0$, and a {\bf hyperbolic space} when $Q$ is of signature $(n,1)$ and $r<0$.
\end{definition}

\begin{definition}
Let us call a {\bf flat geometry} of dimension $n$ a couple $(A,Q)$ consisting of an affine space $A$ of dimension $n$ over $\mathbb{R}$ and a non-degenerate quadratic form $Q$. {\bf Hyperplanes} of the geometry are affine subspaces of dimension $n-1$, while {\bf hypercycles} are the zero sets of the function $p\to Q(\lambda p-c)+r$ where $c\in A$ and $\lambda$ and $r$ are numbers. It is called a {\bf Euclidean space} if $Q$ is of signature $(n,0)$.
\end{definition}

These geometries have a deep connection between them, and this can be formalized in an algebraic way, as done in \HLR{} using geometric algebra. To motivate the definitions in this article, I will show this connection using elementary algebra, similarly to how it is done in \cite{molnar1976}.

Note that in the terminology used throughout this article, hypercycles refer to higher dimensional cycles of all kinds, not specifically to the equidistant curves known from hyperbolic geometry.

\subsection{Flat and curved geometries}

Let us denote by $Q_F$ the quadratic form of a flat geometry and by $Q_C$ the quadratic form of a curved geometry.
In the curved case, all hypercycles are the zero sets of non-homogeneous linear functions, whereas in the flat case, they are not. In the flat case in general, they are of the form $Q_F(\lambda p-c)+r=\lambda^2 Q_F(p)-2\lambda B_F(p,c)+Q_F(c)+r$, where $B_F$ is the associated bilinear form defined as $B_F(u,v)=\frac{1}{2}(Q_F(u+v)-Q_F(u)-Q_F(v))$. Since $B_F$ is non-degenerate, hypercycles are in general of the form $AQ_F(p)+L(p)+R$ where $L$ is a linear form and $R$ and $A$ are constants. To make this into a linear function, we need a new coordinate that is always equal to $Q_F(p)$. We will denote it by $t$. Hence a flat geometry of dimension $n$ can also be described as an affine space $A$ of dimension $n+1$, with the points of the geometry being those $(p,t)\in A$ where $Q_F(p)=t$. These points form a quadric in $A$.

To unify this with the curved case, we may turn to homogeneous coordinates by introducing a new coordinate $z$, taking a projective space $P$ of dimension $n+1$, and introduce a new quadratic form $Q'_F((p,t,z))=Q_F(p)-tz$. Then the points of the geometry are those $q\in P$ where $Q'_F(q)=0$. In this new model, all hyperspaces and hypercycles of the geometry are hyperspaces of $P$, given as the zero sets of homogeneous linear functions.

In the curved cases, the vector space is naturally embedded inside a projective space $V\subset P$ through the introduction of a new coordinate $t$, and $V$ is the complement of the zero set of the function $t=0$. Then the points of the geometry are those $p=(v,t)\in P$ where $Q'_C(p)=Q_C(v)-t^2r=0$, giving us a quadratic form on $P$.
The hyperplanes and hypercycles are the zero sets of homogeneous linear functions.

The main difference between these models is the choice of the set of hyperplanes of the geometry among its hypercycles, which can be described in a simple way. In the curved case, hyperplanes of the geometry are those hyperplanes in $P$ that pass through the origin $\mathcal{O}$ of $V$. In the flat case, hyperplanes are those hypercycles where the expression $Q_F(p)$, or rather the $t$ coordinate of a point, does not appear in their functional form. This happens exactly if the point $(p,t,z)=(0,1,0)$, which we will denote by $\mathcal{O}$, is on the hypercycle.

With this description, a conformal geometry of dimension $n$ is a triple $(P,Q,\mathcal{O})$ consisting of a projective space $P$ of dimension $n+1$, a non-degenerate quadratic form $Q$, and an element $\mathcal{O}\in P$. The spherical, Euclidean and hyperbolic cases are united since $Q$ is of signature $(n,1)$ in all three cases, and the difference is in the norm of the element $\mathcal{O}$: negative in the spherical case, zero in the Euclidean case, and positive in the hyperbolic case. For the case $n=2$, this is summarized in the following table:
\begin{center}
\begin{tabular}{|c|c|c|c|}
\hline
&		Spherical&		Euclidean&		Hyperbolic\\
\hline
$Q$&		$x^2+y^2+z^2-t^2$&	$x^2+y^2-zt$&		$-x^2+y^2+z^2+t^2$\\
\hline
$\mathcal{O}$&	$(0,0,0,1)$&		$(0,0,0,1)$&		$(0,0,0,1)$\\
\hline
$Q(\mathcal{O})$&	$-1$&			$0$&			$1$\\
\hline
\end{tabular}
\end{center}

\subsection{Quadratic lift}

We will start with a simple example.
In two dimensional spherical geometry, there is a vector space $V$ of dimension $3$, and a quadratic form $Q(x,y,z)=x^2+y^2+z^2$. Any line of the geometry (i.e. great circle) is given by a vector $v\in V$ defined up to scalar multiple, and the points on the line are those $p$ where $Q(p)=1$ and $B(v,p)=0$. One would be tempted to define the angle of two such lines $v_1$ and $v_2$ by taking $\arccos B(v_1,v_2)$. But this is ill defined, since $v_1$ and $v_2$ may be multiplied by a constant, hence we need to normalize it:
$$\alpha=\arccos\frac{B(v_1,v_2)}{\sqrt{Q(v_1)Q(v_2)}}$$
To make this definition work, we must restrict the scalars by which we may multiply the the vectors $v_1$ and $v_2$ to be positive. We must also fix the convention of taking the positive square roots. In a more general setting, the square root might not even exist, and this definition is impossible to extend to arbitrary fields, where the sign of elements might not be defined. I will present an approach more in line with abstract linear algebra.

The problem here is that two lines have two different angles at their intersections: $\alpha$ and $\pi-\alpha$, corresponding to different signs chosen for the square root. To resolve such an ambiguity, we could fix the square root by adding another coordinate $w$ to vectors in $V$ whose value is one of $\pm\sqrt{Q(v_1)}$. For this condition to hold, we must extend $Q$ into $Q'(w,x,y,z)=x^2+y^2+z^2-w^2$, and assume that all lines correspond to vectors taken from this quadric. Also, the points are in a natural bijection with elements of the quadric with $w=0$. We get a new bilinear form $B'$, and the scalar product of a point and a line is the same for the new bilinear form as before, $B'(l,p)=B(l,p)$. Then the angle between two lines can be calculated by
$$\alpha=\arccos\left(\frac{B'(v_1,v_2)}{W(v_1)W(v_2)}+1\right)$$
where $W(w,x,y,z)=w$, a linear functional whose zero set are exactly the points.

If we apply this construction to the conformal Euclidean geometry, we get the Lie sphere geometry (\FSDet{}), where the points of the quadric are the points, oriented circles and oriented lines in the Euclidean geometry. By applying it to other geometries, we can extend the Lie sphere geometry to other curved geometries as well.
The main point of this construction is that it creates a duality between points and lines, as the set of either is the intersection of a quadric with some hyperplane.

Since this procedure, which assigns orientations to lines and circles, essentially gives solutions to a quadratic polynomial on the coordinates of the object (in this case, $w^2=Q(v_1)$), I call this the {\bf quadratic lift} of a space. It is essentially the inverse problem of restricting a quadratic form on $V$ to the quotient space $V/v$ for some vector $v\in V$.

It is useful to take a look at the inverse procedure: when $Q(v)\ne 0$, the space $V/v$ can be canonically identified with the subspace $v^\perp$ by orthogonally projecting all the representant vectors of an element in $V/v$ onto $v^\perp$. However, we may also consider the case when $Q(v)=0$. Then there is no canonical identification between $v^\perp$ and $V/v$, but we may take a non-canonical vector $\mathcal{O}$ such that $B(v,\mathcal{O})=-1$ and $Q(\mathcal{O})=0$, and identify $V/v$ with $\mathcal{O}^\perp$. Consider an affine space $A^v=\{x\in\mathcal{O}^\perp\mid B(v,x)=-1\}$. The underlying vector space $V^v\cong\langle v,\mathcal{O}\rangle^\perp$ has a non-degenerate quadratic form $Q^v$ on it, and so for any two vectors $u_1$, $u_2\in A^v$, their difference $u_1-u_2$ is in $V^v$, where it can be evaluated via $Q^v$.

Therefore, even though $V/v$ does not have a natural quadratic structure on it when $Q(v)=0$, taking the projectivization of the vector space $V/v$, denoted by $\mathbb{P}(V/v)$, and identifying $A^v$ by its affine subset, we get a pseudo-Euclidean geometry on it.
By introducing the coordinates $t$ for the coefficient of $v$, and $z$ for the coefficient of $\mathcal{O}$, then the quadratic form can be written as $Q(u)=Q^v(u^v)-tz$ for $u^v$ the orthogonal projection to $V^v$, giving essentially the conformal embedding of a flat geometry $(A^v,Q^v)$ into a vector space.

By applying the quadratic lift to the construction in the previous section, we would be able to identify oriented cycles as elements in a projective space, and define their angles without referencing the square root directly, exactly as in Lie sphere geometry. However, I will instead start at the most general setup of an $n+2$ dimensional projective space, and show what types of homogeneous geometries can be created.

\section{Preliminaries}

We will introduce several essential linear algebraic concepts. Some of them are not defined in the conventional way, such as degenerate quadratic forms.

Let us fix the field $\mathbb{K}$ and a vector space $V$ of finite dimension $n$ over it. When choosing a basis, we will start the indices with $1$. Our main focus shall be the finite fields and the fields $\mathbb{R}$ and $\mathbb{C}$. Often fields of characteristic $2$ behave differently from other fields, and we will have to distinguish them by specifying that $\charop\mathbb{K}\ne2$.

\begin{definition}
A {\bf quadratic form} $Q$ is a homogeneous quadratic polynomial in the coordinates on $V$, or equivalently, the combination of products of linear functions. The direct sum of two quadratic forms $Q_1$ and $Q_2$, defined over $V_1$ and $V_2$, respectively, is a quadratic form on $V_1\oplus V_2$, defined as $(Q_1\oplus Q_2)(v_1\oplus v_2)=Q_1(v_1)+Q_2(v_2)$.
\end{definition}

\begin{definition}
A {\bf bilinear form} $B$ is a polynomial on $V$ in two variables such that it is linear in both of them. Two vectors are {\bf orthogonal}, denoted as $u\perp v$, when $B(u,v)=0$. We denote by $S^\perp=\{u\in V\mid\forall v\in S: B(u,v)=0\}$ the {\bf orthogonal subspace} to the set $S\subseteq V$, and abbreviate $\{v\}^\perp$ by $v^\perp$.
\end{definition}

For every quadratic form $Q$, there is an {\bf associated symmetric bilinear} form $B(u,v)=Q(u+v)-Q(u)-Q(v)$. Conversely, every bilinear form $B$ gives rise to an {\bf associated quadratic form} $Q(v)=B(v,v)$. When $\charop\mathbb{K}\ne2$, these two operations are inverse to each other, up to a multiple of $2$.

\begin{definition}
Let us denote the function $u\to B(v,u)$ by $v^*$. A bilinear form is {\bf non-degenerate} or {\bf regular} if for every non-zero vector $v\in V$, the linear function $v^*$ is non-zero.
\end{definition}

This is equivalent to the following:

\begin{proposition}
Suppose that $\charop\mathbb{K}\ne 2$. Then a bilinear form is non-degenerate if and only if its associated quadratic form is not the direct sum of a quadratic form and the zero form $0(v)=0$.
\end{proposition}

\begin{proof}
If the quadratic form $Q$ is a direct sum $Q_1\oplus Q_2$ over the decomposition $V_1\oplus V_2$, and $Q_2=0$, then for any vector $v\in V_2$, $B(v,.)$ is the constant zero function. Now suppose that there is a vector $v$ such that $B(v,.)$ is the constant zero function. Then for any $u\not\in\langle v\rangle$, $Q(\lambda v+u)=\lambda^2 Q(v)+\lambda B(v,u)+Q(u)=Q(u)$, since $Q(v)=\frac{1}{2}B(v,v)=0$. Hence the quadratic form decomposes, and its component is zero over $\langle v\rangle$.
\end{proof}

To be able to discuss quadratic forms in characteristic $2$ as well, we will define a non-degenerate quadratic form this way, contrary to conventions.

\begin{definition}
A quadratic form $Q$ is {\bf non-degenerate} when there is no $Q'$ such that $Q=Q'\oplus 0$ where $0(v)=0$ is the zero quadratic form.
\end{definition}

The following nomenclature diverges slightly from that found in literature to accomodate the characteristic $2$ case.

\begin{definition}
A non-zero vector $v$ is called {\bf isotropic} if $Q(v)=0$, and anisotropic if $Q(v)\ne0$. It is called {\bf symplectic} if $B(v,v)=0$. It is a {\bf degenerate vector} if $B(v,u)=0$ for all $u\in V$.
\end{definition}

Note that when $\charop\mathbb{K}\ne 2$, isotropic and symplectic vectors are the same, while for $\charop\mathbb{K}=2$, every vector is symplectic.
For a non-degenerate quadratic form, there are no non-degenerate vectors if $\charop\mathbb{K}\ne 2$. When $\charop\mathbb{K}=2$, there might be degenerate vectors, such as the vector $(1,0,0)$ for the quadratic form $x^2+yz$. However, when $\mathbb{K}$ is finite, there can not be two linearly independent degenerate vectors. This is true for an even more general class of fields, called {\bf perfect fields}.

\begin{definition}
A field $\mathbb{K}$ of characteristic $2$ is {\bf perfect} if the polynomial $x^2-u$ for all $u\in\mathbb{K}$ has a root in $\mathbb{K}$.
\end{definition}

Perfect fields are essential to Galois theory for all positive characteristics, however the characteristic $2$ case is the only case we will need in this article. The following theorem from classical Galois theory is included without proof.

\begin{theorem}
Every finite field $\mathbb{K}$ of characteristic $2$ is perfect, and for all $u\in\mathbb{K}$ there is a unique element $u^{1/2}$ such that $(u^{1/2})^2=u$.
\end{theorem}

\begin{lemma}
For a non-degenerate quadratic form over a field $\mathbb{K}$, the degenerate vectors form a subspace. When $\charop\mathbb{K}\ne 2$, this subspace is of dimension $0$. When $\charop\mathbb{K}=2$ and $\mathbb{K}$ is perfect, it is a subspace of dimension at most $1$.
\end{lemma}

\begin{proof}
Consider the subspace of degenerate vectors, $U$. Then for any direct complement $U'$ of $U$, $V$ decomposes as $U\oplus U'$ and $Q$ as $Q|_U\oplus Q|_{U'}$. Clearly $Q|_U$ is non-degenerate. Since all elements of $U$ are orthogonal to each other, the quadratic form can be expressed as $\sum A_i x_i^2$. If $\mathbb{K}$ is perfect, then all the elements $A_i$ admit a unique square root $A_i^{1/2}$, and the form may be written as $\left(\sum A_i^{1/2}x_i\right)^2$, or in some other basis as $y_0^2$. If $\dim U>1$, this is clearly degenerate.
\end{proof}

Consider that since $B(v,u)$ is a linear function, the codimension of the subspace $v^\perp$ is always at most $1$, and it is $0$ only if $v$ is degenerate. Also, for any subspace $U\le V$, the codimension of $U^\perp$ is at most $\dim U$. This gives us the following lemma.

\begin{lemma}
A vector space $V$ contains no degenerate points if and only if $V^\perp=\{0\}$. Suppose that $V$ is such. Then a subspace $U\le V$ satisifies $U^\perp=\{0\}$ if and only if $U=V$.
\end{lemma}

\begin{proof}
In fact, $V^\perp$ is the set of degenerate points. Suppose that $V$ contains no degenerate points. The for any subspace $U$, the codimension of $U^\perp$ is at most $\dim U$. Therefore if $U^\perp=\{0\}$, then $\dim U=\dim V$, and so $U=V$.
\end{proof}

When $\charop\mathbb{K}\ne 2$, every quadratic form has an orthogonal basis, or equivalently, $Q(v)=\sum_{i=1}^na_iv_i^2$ for some $a_i\in\mathbb{K}$.

\begin{notation}
When the characteristic of the field is not $2$, and the quadratic form is of the form $\sum_{i=1} a_iv_i^2$, we will abbreviate this as $[a_1,\dots,a_n]$.
\end{notation}

Any orthogonal set of anisotropic vectors can be extended into an orthogonal basis. However, if we want to include isotropic vectors in a basis, or extend it to the characteristic $2$ case, we need a more general condition than orthogonality. Whenever there is an anisotropic vector $u$ such that $u^\perp\ne V$, we may take another vector $v$ such that $V$ decomposes into the direct sum of $V_1=\langle u,v\rangle$ and $V_2=V_1^\perp$. We may also assume that $Q(v)=0$ and $B(u,v)=1$.

\begin{definition}
A {\bf symplectic couple} will refer to a pair $u$, $v\in V$ of symplectic vectors such that $B(u,v)=1$. A {\bf generalized orthogonal set} $S$ is a set of linearly independent vectors such that $S$ contains several disjoint symplectic couples, and any pair of vectors that are not one of these symplectic couples are orthogonal. A {\bf generalized orthogonal basis} is a generalized orthogonal set that is also a basis.
\end{definition}

\begin{theorem}
Suppose that $V$ contains no degenerate points. Then any generalized orthogonal set may be embedded into a generalized orthogonal basis.
\end{theorem}

\begin{proof}
Consider such a set $S$. We may prove this theorem by induction on the size of $S$.

If $S$ is empty, we may choose any vector and add it to it. If the cardinality of $S$ is equal to $\dim V$, then we are done.

If there is a non-symplectic vector $v\in S$, and $S'=S\setminus\{v\}$, then $S'$ is an orthogonal set in $v^\perp$, satisfying the conditions of the theorem, and we may extend it by induction.
Similarly if there is a symplectic couple $v_1$, $v_2\in S$, and $S'=S\setminus\{v_1,v_2\}$, then $S'$ is an orthogonal set in $\langle v_1,v_2\rangle^\perp$.

Now let us suppose that $S$ contains only pairwise orthogonal symplectic vectors. Take one such vector $v\in S$ and consider $S'=S\setminus\{v\}$. Since these are all linearly independent vectors, $v^\perp$ does not contain the entirety of $S'^\perp$, which is of dimension $n-|S'|$. Hence there is a vector $u\in S'^\perp$ such that $B(u,v)=1$. Then the space $\langle u,v\rangle^\perp$ has no non-degenerate vectors, and $S'$ can be extended by induction.
\end{proof}

\begin{lemma}
If $\charop\mathbb{K}=2$ and $\mathbb{K}$ is perfect, then there are degenerate points for a non-degenerate quadratic form if and only if $2\nmid\dim V$, and in that case they form a subspace of dimension $1$.
\end{lemma}

\begin{proof}
First recall that the subspace of degenerate vectors may have dimension at most $1$.
Since in characteristic $2$ all vectors are symplectic, by the previous lemma the dimension of a vector space with no degenerate points can only be even, and so odd dimensional vector spaces must contain degenerate vectors, which form a subspace of dimension $1$.

On the other hand, by taking any direct complement of the subspace of degenerate vectors, the quadratic form on that subspace is not degenerate and it has no degenerate vectors. Hence it must be of even dimension. However, the dimension of the subspace of degenerate vectors is at most $1$, and so even dimensional vector spaces do not contain degenerate vectors.
\end{proof}

\begin{definition}
An {\bf isometry} $f$ between two spaces $(U,Q_U)$ and $(V,Q_V)$ is an injective linear map that is compatible with the quadratic forms: $Q_V(f(u))=Q_U(u)$. Two spaces with a bijective isometry between them are {\bf isometric}.
\end{definition}

Let us consider \Kit{}, {\bf Corollary 1.2.1.}:

\begin{theorem}
Given two subspaces $U$, $V\le W$ of the same dimension such that neither of them contains degenerate vectors, any isometry between $U$ and $V$ extends into an isometry of $W$.
\end{theorem}

This gives the following result:

\begin{corollary}
If $Q$ is non-degenerate, then the orbits of non-degenerate vectors are exactly the non-degenerate vectors of the same norm: $\{v\in V\mid Q(v)=\alpha, v\not\in V^\perp\}$. Furthermore, for any isometric subspaces $U_1$, $U_2$ that contain no degenerate vectors, the subspaces $U_1^\perp$ and $U_2^\perp$ are also isometric.

If either $\charop\mathbb{K}\ne 2$ or the field is perfect and $2\mid\dim V$, there are no degenerate vectors, and the orbits are exactly the vectors of the same norm.
\end{corollary}

In \Kit{}, the following is referred to as a hyperbolic subspace. To avoid confusion with hyperbolic geometry, we will refer to it as a symplectic subspace.

\begin{definition}
A {\bf symplectic subspace} of dimension $2n$ is a vector space with a quadratic form that may be written as $\sum_i x_{2i}x_{2i+1}$ in some basis. In the case of $\charop\mathbb{K}\ne 2$, this is equivalently to the restriction of the quadratic form to the subspace being $[+1,\dots,+1,-1,\dots,-1]$ with the same amount of $+1$ as $-1$.
\end{definition}

\begin{lemma}
Given a symplectic space $V$ and a value $\lambda\in\mathbb{K}$, there is a vector $v\in V$ such that $Q(v)=\lambda$.
\end{lemma}

\begin{proof}
Assume that the quadratic form is $xy$, and choose the coordinates $v=(\lambda,1)$.
\end{proof}

Let us denote the projectivization of $V$ by $\mathbb{P}V$.
When we consider the projective space $\mathbb{P}V$, it is still meaningful to talk about isotropic and anisotropic vectors, as well as orthogonal and non-orthogonal pairs of vectors.

\begin{notation}
The zero set of $Q$ in $\mathbb{P}V$ shall be denoted by $[[Q]]$.
\end{notation}

Also, even though we can not evaluate $Q$ on a point $p\in\mathbb{P}V$, the set of values it can take for some $p$ is restricted: when $[v]=p$ is replaced by $\lambda v$, we get the norm $Q(\lambda v)=\lambda^2 Q(v)$. Hence the value $Q(p)$ is identified up to multiplying by squares:

\begin{definition}
Given an element $p\in\mathbb{P}V$, its {\bf norm} $Q(p)$ is an element of the group $\mathbb{K}^\times/(\mathbb{K}^\times)^2\cup\{0\}$ (where $\mathbb{K}^\times$ is the multiplicative subgroup of $\mathbb{K}$), defined as the equivalence class that $Q(v)$ may take for $[v]=p$. Two points $p_1$ and $p_2$ are {\bf orthogonal} if their representants in $V$ are orthogonal.
\end{definition}

\begin{definition}
For a point $p\in\mathbb{P}V$, we will define $\mathbb{P}V/p$ as the projective space $\mathbb{P}(V/v)$ for some representant $v\in V$ for $p$.
\end{definition}

Finally, let us review a few important but simple properties of subspaces and quotient spaces.

\begin{lemma}
The non-degenerate bilinear form $B$ gives rise to a duality between the spaces $V/p$ and $p^\perp$. That is, there is a non-degenerate pairing $B:V/p\times p^\perp\to\mathbb{K}$.
\end{lemma}

\begin{proof}
Take a pair of elements $[a]\in V/p$ as the image of $a\in V$, and $b\in p^\perp$. Since $[a]$ consists of elements $a+\lambda p$ for $\lambda\in\mathbb{K}$, and $B(a+\lambda p,b)=B(a,b)+\lambda B(p,b)=B(a,b)$, the bilinear form is well defined.
\end{proof}

This gives a natural isomorphism between the spaces $(p^\perp)^*\cong V/p$.

\begin{notation}
Let us denote the natural quadratic form on $p^\perp$ arising via $Q$ by $Q^p$.
\end{notation}

\begin{lemma}
When $B(p,p)\ne 0$, the spaces $V/p$ and $p^\perp$ may be identified via a projection through $p$.
\end{lemma}

\begin{proof}
For any vector $u\in V$, we may take $u-\frac{B(p,u)}{Q(p)}p\in p^\perp$, and this gives a map $\mathbb{P}V/p\to p^\perp$.
\end{proof}

\section{Universal conformal geometry}

\subsection{Definitions}

First we will introduce a few definitions, then look at some examples. The terminology resembles that found in \FSDet{}.
It is useful to compare these definitions with examples given in {\bf 4.2}.

Consider a vector space $V$ over $\mathbb{K}$ where $Q$ is a non-degenerate quadratic form on $V$. Recall that there is a corresponding bilinear form $B(u,v)=Q(u+v)-Q(u)-Q(v)$, and two vectors $u$, $v\in V$ are orthogonal if $B(u,v)=0$. Two points in $\mathbb{P}V$, the projectivization of $V$, are orthogonal if their representing vectors are orthogonal, which is well-defined. Also, the zero set of $Q$ in $\mathbb{P}V$ is well-defined.

\begin{definition}
A {\bf universal conformal geometry} of dimension $n$ is a tuple $(V,Q,P,L)$ where $V$ is a vector space of dimension $n+3$ over $\mathbb{K}$, $Q$ is a non-degenerate quadratic form with non-degenerate associated bilinear form $B$, and $P$, $L\in\mathbb{P}V$ are orthogonal.
\end{definition}

The orthogonality of $P$ and $L$ is important for certain essential calculations, such as the incidence structure of hyperplanes in the geometry, but some of the results hold even without this hypothesis. Note that $Q$ may be multiplied by any scalar without altering the isomorphism class of such a geometry.

\begin{definition}
The set $[[Q]]$ of points in $\mathbb{P}V$ where $Q$ is zero is called the {\bf Lie quadric}. The geometry is {\bf empty} if $[[Q]]$ is empty, and {\bf non-empty} otherwise.
\end{definition}

\begin{definition}
The {\bf dual geometry} is the same tuple with $P$ and $L$ exchanged.
\end{definition}

\begin{definition}
An {\bf (oriented) hypercycle} is an element $c\in\mathbb{P}V$ such that $Q(c)=0$. A hypercycle $c$ is a {\bf pointcycle} or {\bf point} if $B(P,c)=0$, and an {\bf (oriented) hyperplane} or {\bf hyperplanecycle} if $B(L,c)=0$. A hypercycle that is both a point and a hyperplane is {\bf ideal} (or {\bf ideal hyperplane}). When the dimension is two, a hyperplane is called a {\bf line}.
\end{definition}

\begin{definition}
Two hypercycles $c_1$, $c_2$ are {\bf touching} or {\bf incident} if $B(c_1,c_2)=0$. The geometry is {\bf non-degenerate} if there is a pair of an incident non-ideal point and non-ideal hyperplane. For any hypercycle, $[[c]]$ will denote the {\bf set of points} incident to the hypercycle, i.e. $[[Q]]\cap P^\perp\cap c^\perp=[[Q]]\cap\langle P,c\rangle$.
\end{definition}

Let us fix a representant of $L$ and $P$ in the vector space $V$. The following definition is a verbatim copy of the definition in \FSDet{}, and can be used in measuring distances and angles.

\begin{definition}
For two cycles $c_1$ and $c_2$, let the {\bf relative power} be $\frac{B(c_1,c_2)}{B(c_1,P)B(c_2,P)}$ and the {\bf inversive separation} be $\frac{B(c_1,c_2)}{B(c_1,L)B(c_2,L)}$.
\end{definition}

Since all the points appear in the subspace $P^\perp$, it is useful to consider the restriction of $Q$ to this space.
We will distinguish between subspaces of $P^\perp$ that are the points of a hypercycle, and those that are not, by calling the former {\it actual} hypercycles, while either is a {\it virtual} hypercycle.

\def\pointspace{%
\begin{definition}
The {\bf pointspace} is the orthogonal subspace $P^\perp\subset\mathbb{P}V$. The elements of the dual space $(P^\perp)^*\cong \mathbb{P}V/P$ are called {\bf unoriented virtual hypercycles}, with those that are the projection of an oriented hypercycle in $\mathbb{P}V$ to $\mathbb{P}V/P$ called an {\bf unoriented (actual) hypercycle}. The restriction of $Q$ to $P^\perp$ is $Q^P$, and two vectors $x_1$ and $x_2$ are {\bf incident} in the pointspace if $B^P(x_1,x_2)=0$.
\end{definition}

Note that $L$ is contained within the pointspace since $P\perp L$, and that $\mathbb{P}V/P\cong P^\perp$ in a natural way when $B(P,P)\ne0$.%
}

\pointspace

In general, any oriented hypercycle has an associated unoriented hypercycle, constructed as its projection to $\mathbb{P}V/P$, while the converse is not true. The set of points of two oriented hypercycles with identical associated unoriented hypercycles is the same. For two geometries differing only in which unoriented hypercycles emerge as projections of oriented hypercycles, I will refer to them as cycle equivalent. This can be formalized in a more compact manner.

\begin{definition}
Two geometries are {\bf cycle equivalent} if the tuples $(P^\perp, Q^P, L)$ are isomorphic.
\end{definition}

\subsection{Examples over the reals}

A non-degenerate quadratic form with signature of the form $(k,l)$ means that the vector space decomposes to two orthogonal subspaces of dimension $k$ and $l$ with $Q$ positive definite on the first and negative definite on the second.

Since $P$ and $L$ are orthogonal, they may be part of a generalized orthogonal basis. Vectors inside the $n+3$ dimensional vector spaces may be denoted by triplets ${\bf v}=(\overline v,v_{n+2},v_{n+3})$ with $\overline v$ in an $n+1$ dimensional vector space, or $(\overline v,v_{n+1},v_{n+2},v_{n+3})$ with $\overline v$ in an $n$ dimensional vector space.

In this section, we will use the conventional definition of the bilinear form as $B(u,v):=\frac{1}{2}(Q(u+v)-Q(u)-Q(v))$ unlike in the rest of this article where the scalar $\frac{1}{2}$ has been left off. This will give the inversive separation and relative powers here as double the previous definitions.

The following examples are given, without proofs, to illustrate the definitions in {\bf 4.1}, however they can be checked easily.

\medskip

$\bullet$ {\bf Elliptic geometry:} The $n$ dimensional elliptic geometry $\mathcal{E}^n$ consists of a quadratic form of signature $(n+1,2)$, with both $P$ and $L$ having a negative norm. When represented as $L=[e_{n+2}]$ and $P=[e_{n+3}]$, the points take coordinates $(\overline c,1,0)$, where $\overline c$ are its coordinates in the spherical model. A cycle of center $\overline c$ and radius $\rho$ has coordinates $(\overline c,\cos(\pm\rho),\sin(\pm\rho))$. Lines polar to a point $\overline c$ have coordinates $(\overline c,0,\pm1)$.

The inversive separation of two points of normalized distance $\delta$ is $\cos\delta-1$. The relative power of two cycles intersecting at angle $\vartheta$ is $\cos\vartheta-1$.

\medskip

$\bullet$ {\bf Hyperbolic geometry:} The $n$ dimensional hyperbolic geometry $\mathcal{H}^n$ consists of a quadratic form of signature $(n+1,2)$, with $P$ having negative and $L$ positive norm. When represented as $L=[e_{n+2}]$ and $P=[e_{n+3}]$, the points take coordinates $(\overline c,1,0)$, where $\overline c$ are its coordinates in the hyperboloid model, where conventionally one of the base vectors has a negative norm, for instance $e_{n+1}$. A cycle of center $\overline c$ and radius $\mathfrak{r}$ has coordinates $(\overline c,\cosh(\pm\mathfrak{r}),\sinh(\pm\mathfrak{r}))$. A hyperplane whose normalized vector in the hyperboloid model is $\overline l$ gives $(\overline l,0,1)$. Hypercycles at constant $\mathfrak{d}$ distance from a line whose normalized vector is $\overline l$ is $(\overline l,\sinh(\pm\mathfrak{d}),\cosh(\pm\mathfrak{d}))$. Paracycles centered on the ideal point $u$ are of the form $(u,\lambda,\lambda)$ for some parameter $\lambda$.

The inversive separation of two points of normalized distance $\mathfrak{d}$ is $1-\cosh\mathfrak{d}$. The relative power of two cycles intersecting at angle $\vartheta$ is $\cos\vartheta-1$.

\medskip

$\bullet$ {\bf Parabolic geometry:} The $n$ dimensional parabolic geometry $\mathcal{P}^n$ consists of a quadratic form of signature $(n+1,2)$, with $P$ having negative norm and $L$ being isotropic. Let us choose a non-orthogonal isotropic partner $O$ to $L$, that is also orthogonal to $P$. When represented as $L=[e_{n+1}]$, $O=[e_{n+2}]$, $P=[e_{n+3}]$, the points take coordinates $(\overline c,-Q(\overline c),1,0)$, with $Q$ being the restriction of $Q$ onto $\langle O,L,P\rangle^\perp$. Here, $\overline c$ gives the coordinates in an affine space and $O$ corresponds to the choice of origin. Then a cycle of center $\overline c$ and radius $r$ has coordinates $(\overline c,(\pm r)^2-Q(\overline c),1,\pm r)$. A hyperplane whose equation is $B(\overline c,\overline l)=d$ with $\overline l$ normalized, takes the coordinates $(\overline l,-2d,0,1)$.

The inversive separation of two points of normalized distance $d$ is $-\frac{1}{2}d^2$. The relative power of two cycles intersecting at angle $\vartheta$ is $\cos\vartheta-1$.

\medskip

$\bullet$ {\bf The Minkowski plane:} The $2$ dimensional Minkowski plane $\mathcal{M}^2$ consists of a quadratic form of signature $(3,2)$, with $P$ having positive norm and $L$ being isotropic. Let us choose a non-orthogonal isotropic partner $O$ to $L$, that is also orthogonal to $P$. When represented as $L=[e_3]$, $O=[e_4]$, $P=[e_5]$, the points take coordinates $(\overline c,-Q(\overline c),1,0)$, with $Q$ being the restriction of $Q$ onto $\langle O,L,P\rangle^\perp$. Here, $\overline c$ give the coordinates in an affine space and $O$ corresponds to the choice of origin. The cycles are certain hyperbol\ae{} that have coordinates $(\overline c,(\pm r)^2-Q(\overline c),1,(\pm r))$, while the lines are $(\overline l,-2d,0,1)$ where $\overline l$ is of norm $1$.

This only gives half of all the lines and origin-centered hyperbol\ae{} in the Minkowski plane. However, the other set of lines and hyperbol\ae{}, obtained by making the norm of $P$ negative, give a geometry that is isomorphic to the Minkowski plane.

\medskip

$\bullet$ {\bf The de Sitter surface:} The de Sitter surface is the dual Hyperbolic geometry $\mathcal{H}^*$, that has $P$ be positive and $L$ be negative. Although it is trivial to derive from the Hyperbolic geometry, this description will help for the anti-de Sitter surface. Taking $L=[e_4]$ and $P=[e_5]$, the points take coordinates $(\overline c,1,0)$, where $\overline c$ are the points in the hyperboloid model that have positive norm instead of the usual negative. Cycles then have the form $(\overline c,\sinh\mathfrak{r},\cosh\mathfrak{r})$ while lines $(\overline l,0,1)$ where $\overline c$ has negative norm and $\overline l$ has positive norm.

\medskip

$\bullet$ {\bf The anti-de Sitter surface:} It is similar to the de Sitter surface, but $L$ is positive as well. Then the points, cycles and lines have the same form as in the de Sitter case, but for cycles, $\overline c$ has positive norm, and for lines, $\overline l$ has negative norm.

\medskip

$\bullet$ {\bf The Laguerre/Galilei plane:} (\cite{yaglom1969}, \cite{yaglom1979}, \cite{yaglom1981}) Both $P$ and $L$ are isotropic. Since they are orthogonal, we have to choose an isotropic partner for each. Let $O$ be the partner of $L$ and $D$ that of $P$. Suppose that $P=[e_2]$, $O=[e_3]$, $L=[e_4]$ and $D=[e_5]$. Then the points take coordinates $(x,y,1,-x^2,0)$. Cycles are sets of the form $Ax+B+Cy-Dx^2=0$, which are parabol\ae{} whose axis of symmetry is parallel to $P$. Lines then take the form $(u,v,0,w,t)$ with $u^2-tv=0$, and contain the points where $2ux-yt-w=0$. The only lines that are not described this way are those where $t=0$ but $u\ne 0$, that is those that are parallel to $P$.

\medskip

\begin{remark}
In all of these models, every cycle may appear twice, with a positive radius and a negative radius. They are incident if they are tangent as subsurfaces, with compatible orientation.
\end{remark}

\subsection{Classification}

\begin{theorem}
Suppose that $\charop\mathbb{K}\ne 2$. A geometry is non-empty if and only if $Q(P)\ne 0$ and there is a subspace with the quadratic form $[Q(P),+1,-1]$, or $Q(P)=0$ and has a subspace with the quadratic form $[+1,+1,-1,-1]$.
\end{theorem}

\begin{proof}
If $Q(P)\ne 0$, the quadratic form restricted to the space $P^\perp$ is non-degenerate. Since it contains an isotropic vector, $P^\perp$ contains a subspace with the quadratic form $[+1,-1]$ that extends to $[Q(P),+1,-1]$.

If $Q(P)=0$, then $P$ and $v$ are orthogonal isotropic vectors, hence they can be embedded into a subspace with the quadratic form $[+1,+1,-1,-1]$.

In the other direction, suppose that we have a subspace of the prescribed form. Then there is a vector $P'$ with the same norm as $P$, and an element $v\perp P'$ with $Q(v)=0$. Since the isometry group is transitive, we can send $P'$ to $P$, and the image of $v$ is a point of the geometry.
\end{proof}

\begin{theorem}
A geometry is non-degenerate if and only if it contains a symplectic subspace of dimension $4$.
\end{theorem}

\begin{proof}
A non-degenerate geometry contains an incident pair of a point $p$ and a hyperplane $l$, both non-ideal. Since $Q(p)=Q(l)=0$ and $P\perp p$, $L\perp l$, the subspace $\langle P,L,p,l\rangle$ decomposes as two orthogonal subspaces $\langle P,l\rangle\oplus\langle L,p\rangle$. Since $p$ and $l$ are non-ideal, these two subspaces are symplectic. Therefore there is a symplectic subspace of dimension $4$, giving the desired form for the quadratic form.

Now if there is a symplectic subspace of dimension $4$, by decomposing it as the direct sum of two symplectic subspaces of dimension $2$ each, one can find orthogonal vectors $P'$ and $L'$ of norms $Q(P)$ and $Q(L)$ in it, and isotropic vectors $p$ and $l$ such that $P'\perp p\perp l\perp L'$. Since there are no degenerate vectors in the vector space, there is an isometry $f$ that maps $\langle P,L\rangle$ to the isomorphic subspace $\langle P',L'\rangle$, and $(f(p),f(l))$ gives an incident pair of a point and a hyperplane.
\end{proof}

In the following sections we will classify the possible geometries. It is enough to specify the equivalence class of a quadratic form (see \cite{kitaoka1993}) up to scalar multiple, and the norms of $P$ and $L$, since two different couples with $P'$ and $L'$ having the same norms admit an isometry that sends the subspace $\langle P,L\rangle$ into $\langle P',L'\rangle$. We will denote by $n$ the dimension of $V$.

\subsubsection{Classification for quadratically closed fields}

When $\mathbb{K}$ is quadratically closed (and $\charop\mathbb{K}\ne 2$), the quadratic form has a unique form, and it contains a symplectic subspace of maximal dimension ($n$ or $n-1$, whichever is even), hence the geometry is non-degenerate if $n\ge 4$. The norms of $P$ and $L$ are either $0$ or $1$ up to a square, hence in each dimension, we have $4$ geometries.

This includes the field of complex numbers, $\mathbb{C}$. Note that we are not studying Hermitian forms, as those are defined using conjugation, and so they are not quadratic forms. Since conjugation is an $\mathbb{R}$-linear operation, Hermitian forms might be better classified as geometries over the $\mathbb{R}$ algebra $\mathbb{C}$.

\subsubsection{Classification for reals}

For $\mathbb{K}=\mathbb{R}$, the quadratic form is determined by its signature. The forms of signature $(k,l)$ and $(l,k)$ give the same geometries, so let us suppose that $k\ge l$. The dimension of a maximal symplectic subspace is then $2l$, so the geometry is non-degenerate if and only if $l\ge 2$.

The norms of $P$ and $L$ are either $1$, $0$ or $-1$. Unless $k=l$ or $Q(P)=Q(L)=0$, the pair $(Q(P),Q(L))$ determines uniquely the geometry. However, when $k=l$ and either $Q(P)\ne 0$ or $Q(L)\ne 0$, we can flip the signs of $Q(P)$ and $Q(L)$.

In dimension $d=n-3$, not counting difference in sign, we can choose $l$ to be $2\le l\le d+1$, and $P$ and $L$ may be chosen in $9$ different ways, giving $9d/2$ geometries when $d$ is even, and $(9d-1)/2$ when $d$ is odd.

\subsubsection{Classification for finite fields}

Suppose that $\charop\mathbb{F}_q\ne 2$ is a finite field of $q$ elements, and choose a non-square ${\bf e}$. Every element of the field is either a square, or a square multiplied by ${\bf e}$. A quadratic form written as $\sum a_ix_i^2$ is determined by its determinant, $\det Q=\prod a_i$, up to square, and hence there are two non-isomorphic quadratic forms in each dimension (\Kit{}). A maximal symplectic subspace is always of dimension at least $n-1$ or $n-2$, whichever is even. Therefore the geometry is non-degenerate if and only if $n\ge 5$ or $n=4$ and the vector space is symplectic.

When the dimension of the geometry ($n-3$) is even, the two quadratic forms can be identified by multiplying one by ${\bf e}$, hence we may fix one of the two. Both the norms of $P$ and $L$ may be chosen as either $0$, $1$ or ${\bf e}$, giving $9$ geometries.
When the dimension is odd, the two quadratic forms are separate, and multiplying it by ${\bf e}$ does not change its isomorphism class. Only one of them is symplectic, so if $n=4$, we must choose that one, while if $n\ge 5$, we may choose either. Unless both $P$ and $L$ have norm $0$, multiplying $Q$ by ${\bf e}$ gives different values for $P$ and $L$, giving $5$ possibilities for the pair $(Q(P),Q(L))$. Therefore there are $10$ geometries if $n>5$, but only $5$ if $n=4$.

\subsubsection{Partial classification for finite fields in characteristic $2$}

Suppose that the field has characteristic $2$. In a perfect field every polynomial $x^2=a$ has a unique solution $a^{1/2}$. A quadratic polynomial of the form $x^2+Ax+B$ with $A\ne 0$ may always be written as $A^2((x/A)^2+(x/A)+B/A^2)$, and solving such a polynomial may be reduced to solving polynomials of the form $x^2+x+C$. The map $x\to x^2+x$ is additive, though not linear, and sends each pair $u$ and $u+1$ to the same value. Hence in a finite field, there is an element ${\bf e}$ such that every element in of the field is either of the form $\tau^2+\tau$ or ${\bf e}+\tau^2+\tau$.

Let us consider a finite field of characteristic $2$, which is a perfect field.
Here we will only consider geometries with non-degenerate bilinear forms, which only happens when the vector space is of even dimension, hence the geometry is of odd dimension. Every quadratic form may be written as $\sum (x_{2i}^2+x_{2i}x_{2i+1}+a_ix_{2i+1}^2)$. Every $a_i$ may be replaced by $a_i+\tau^2+\tau$ for some $\tau$ by an appropriate change of basis in $x_{2i}$ and $x_{2i+1}$, and the quadratic form is determined by $\sum a_i$ up to some term $\tau^2+\tau$ (this is the Arf invariant, see \Kit{}, {\bf Chapter 1.3}). A maximal symplectic subspace is always of dimension at least $n-2$, and so if $n\ge 6$, the geometry is non-degenerate.

Suppose that the dimension of the geometry ($n-3$) is at least $3$. There are two non-isomorphic quadratic forms. The norms of $P$ and $L$ can either be $0$ or $1$, giving us $4$ separate possibilities.

\section{Measurements}

\subsection{The pointspace}

Recall the definition of the pointspace:

\pointspace

\medskip

When $P$ is anisotropic and $\charop\mathbb{K}\ne 2$, any hypercycle can be uniquely projected onto $P^\perp$ by connecting them with a projective line in $\mathbb{P}V$ to $P$ and taking their intersection on $P^\perp$. There is a strong connection between incidence in the pointspace and the complete geometry.

\begin{definition}
For a vector $x\in P^\perp$, let us denote by $[[x]]^P$ the set $[[Q^P]]\cap x^\perp$.
\end{definition}

\begin{lemma}
Suppose that $B(P,P)\ne0$, and consider the restriction of $Q$ onto $P^\perp$, $Q^P$. Take a hypercycle $c$ in $\mathbb{P}V$ and consider $[[c]]=[[Q]]\cap P^\perp\cap c^\perp$. Then if $c^P$ is the projection of $c$ onto $P^\perp$, the set $[[c^P]]^P=[[Q^P]]\cap (c^P)^\perp$ is equivalent to $[[c]]$.
\end{lemma}

\begin{proof}
This is a simple consequence of the fact that given a point $p\perp P$, we have $Q^P(p)=Q(p)$, and since $c^P$ is a linear combination of $c$ and $P$, $B(p,c)=B(p,c^P)$.
\end{proof}

Usually there are two oriented cycles for each unoriented cycle as their image, which in the model geometries correspond to two orientations of the same cycle.

\begin{definition}
The image of oriented hypercycles under the projection onto $\mathbb{P}V/P$ are (unoriented) {\bf (actual) hypercycles} and those of oriented hyperplanes are (unoriented) {\bf (actual) hyperplanes}.
\end{definition}

We will only consider unoriented hypercycles in this section, and we will use the convention that hypercycles are unoriented actual hypercycles, unless we call them virtual specifically.

In order to discuss cycles of fewer dimensions, we need the following definitions:

\begin{definition}
A {\bf virtual subcycle} $S$ of dimension $k$ is a subspace of dimension $k+1$ of the projective space $P^\perp$. It is an {\bf (actual) subcycle} $S$ if it is identified by a set of oriented hypercycles $c_1$, {\dots}, $c_{n-k}$ with $n$ being the dimension of the geometry, as the set $S=\langle P,c_1,{\dots},c_{n-k}\rangle^\perp$. When all the $c_i$ are hyperplanes, the subspace is called a {\bf subplane}. Its codimension is $n-k$. Subplanes of dimension $1$ are called {\bf lines}.
\end{definition}

Note that an alternative formulation for a subplane, considering $P\perp L$, is a subcycle that contains $L$.

\begin{remark}
In general, a distinction must be made between virtual subcycles and actual subcycles, as not all virtual subcycles are actual. In particular, which virtual subcycles admit a set of oriented hypercycles depends on the norm of $P$. However, geometries with isomorphic pointspaces have isomorphic virtual subcycles, and only the selection of actual subcycles changes by non-identical models. These models differ only in the norm of $P$, and are referred to as cycle equivalent.
\end{remark}

\begin{example}
Consider the Minkowski space over the reals, whose lines are generally divided into space-like and time-like lines, depending on the norm of their normal vectors. If the pointspace of some geometry is isomorphic to such a Minkowski space, only one of these two sets may be actual subcycles, as the projections of such cycles will have different signs for space-like and time-like vectors.
\end{example}

Virtual subcycles of codimension $1$ are in a bijection with virtual hypercycles $c\in \mathbb{P}V/P$. We will not distinguish between these two in terminology, unless needed. In particular, in dimension $2$, both a hyperplane $l$ and the subspace $\langle P,l\rangle$ are called a line.

\begin{lemma}
Consider a hypercycle $c$ and the restriction of the quadratic form to the generated subcycle $\langle P,c\rangle^\perp$. It is non-degenerate if and only if $B(P,c)\ne 0$. In particular, if $c$ is a hyperplane, it is non-degenerate if and only if $c$ is non-ideal.
\end{lemma}

\begin{proof}
If $B(P,c)\ne 0$, then $\langle P,c\rangle$ is non-denegerate, hence its orthogonal subspace is non-degenerate as well. If $B(P,c)=0$, since $Q(c)=0$ as well, $\langle P,c\rangle$ is degenerate. Therefore if its orthogonal subspace were not degenerate, this would lead to a contradiction.
\end{proof}

This lemma gives a more general property for ideal hyperplanes:

\begin{definition}
A virtual subplane is {\bf quasi-ideal} if the quadratic form restricted to it is degenerate.
\end{definition}

\begin{remark}
It can be verified that in the Euclidean, spherical and hyperbolic models, quasi-ideal subplanes are those that contain no non-ideal points.
\end{remark}

\subsection{Incidence}

For the sake of completeness, we will review some essential information about the incidence structure of these geometries. In this section, every hypercycle is unoriented. We will refer to a set of elements of $\mathbb{P}V$ as {\bf independent} if the representant vectors in $V$ are linearly independent.

Any points $p$, $q\in[[Q^P]]$ in the pointspace that are collinear with $L$ are contained in exactly the same hyperplanes. This is a more general case of the antipodal points in spherical geometry, which motivates the following definition:

\begin{definition}
Any pair of points in the pointspace that are collinear with $L$ are called {\bf antipodal}.
\end{definition}

\begin{theorem}
Any $k$ independent points define a virtual subcycle of dimension $k-2$. Also, any $k$ independent points where no pair are antipodal define a virtual subplane of dimension $k-1$.
Therefore $n$ points with no antipodal pairs have at most one hyperplane incident to them all.
\end{theorem}

\begin{proof}
Any $k$ elements of the projective space $P^\perp$ give a subspace of dimension $k-1$, which defines a virtual subcycle of dimension $k-2$. Since $L\in P^\perp$, $k$ independent points and $L$ define a virtual subcycle of dimension $k-1$ containing $L$, hence a virtual subplane.
\end{proof}

\begin{theorem}
The intersection of $k$ independent virtual hyperplanes is a virtual subplane of dimension $n-k$. In particular, if $n=k$, this is of dimension $0$, and contains at most two antipodal points.
\end{theorem}

\begin{proof}
The first part is a tautological consequence of the dimensions of the subspaces. Any subplane of dimension $0$ is by definition a subspace of $\mathbb{P}V$ of dimension $1$ containing $L$. Such a subspace intersects $[[Q]]$, a quadratic surface, in at most $2$ points, and if they are elements of the pointspace, they must be antipodal, since the subspace contains $L$.
\end{proof}

In the spherical case, when we identify points on the same line passing through the origin, we get the elliptical geometry, where two lines have a single intersection. In the general case, we may identify points on lines passing through $L$ (i.e. $P^\perp/L$). Finally, for the sake of completeness, we may consider the following generalizations of the Poincar\'e and Cayley-Klein models, defined as an extension to the definitions in \FSDet{} (also studied in \HLR{}).

\begin{definition}
The {\bf Poincar\'e/inversive model} is the space $P^\perp$ with $Q$, with the points being the zero set of $Q$ restricted to $P^\perp$. The elements of $P^\perp$ are the cycles.
The {\bf Cayley-Klein/projective model} of the geometry is the projective space $P^\perp/L$. This identifies the points on the subplanes of dimension $1$.
\end{definition}

\begin{remark}
Although these models are interesting in relationship to the geometry, inversive models can also be studied on their own, as projective spaces with a quadratic form defined on them. They have a natural M\"obius structure on them, with the points corresponding to points of norm zero, and cycles as the intersections of subspaces of the projective space with the set of points. When the projective space is of dimension $3$, choosing a non-degenerate quadratic form gives us a {\bf Miquelian plane}. See \cite{chen1970}.
\end{remark}

\subsection{Distance}

Distance is the set of orbits of pairs of points under the action of isometries. As such, it is enough to consider the pointspace $P^\perp$ when defining distance. Furthermore, distance should be an additive property, in the sense that if given three points $A$, $B$, $C$ on a line, $d(A,C)$ is either the sum or difference of $d(A,B)$ and $d(B,C)$. Hence we will first define distance on a fixed line.

Let us fix a line $\mathcal{L}\subset P^\perp$ (i.e. a two dimensional subspace of $\mathbb{P}V$ containing $L$). It is not quasi-ideal if the restriction $\mathcal{Q}:=Q|_\mathcal{L}$ of the quadratic form $Q$ is non-degenerate. We will refer to the zero set of $\mathcal{Q}$ as the {\bf points of the line} and denote it by $[[\mathcal{L}]]$.

\begin{theorem}
Suppose that $\charop\mathbb{K}\ne 2$, and let $\widetilde\Gamma$ denote the group of automorphisms of $\mathcal{L}$ preserving $\mathcal{Q}$ and $L\in\mathcal{L}$. There is always a natural index $2$ subgroup $\Gamma$ that acts freely transitively on the non-ideal points $[[\mathcal{L}]]$ of the line. These groups depend only on the isomorphism class of the quadratic form on $L^\perp$.
\end{theorem}

\begin{proof}
Consider that all isometries fix $L\in\mathcal{L}$. Take the underlying vector space of $\mathcal{L}$, $U$, and we may identify $L$ with one of its representant vectors in $U$. The automorphisms of the projective space $\mathcal{L}$ that preserve $L\in\mathcal{L}$ can be identified with automorphisms of the vector space $U$ that preserve the vector $L\in U$, since the only degree of freedom is a scalar multiple, which is fixed by fixing the vector $L$.

First we will show that $\widetilde\Gamma$ acts transitively. Consider a pair of isotropic vectors $p$ and $p'$ in $U$ corresponding to non-ideal points, i.e. $B(L,p)\ne 0$ and $B(L,p')\ne 0$. We may assume that $B(L,p)=B(L,p')=1$ by taking an appropriate scalar multiple of $p$ and $p'$. Then the subspaces $\langle L,p\rangle$ and $\langle L,p'\rangle$ are isometric, and there is an isometry that sends one into the other. We may also assume that it fixes $L$, since the space $\langle L,p'\rangle$ is non-degenerate, and $L$ can be sent to any vector of the same norm via an isometry of $\langle L,p'\rangle$.

Now let us fix an isotropic vector $p$ such that $B(L,p)\ne 0$ and consider its stabilizer. Since $\langle L,p\rangle$ is non-degenerate, its orthogonal complement $U':=\langle L,p\rangle^\perp$ is also non-degenerate and of dimension $1$. The isometries of such a space are multiplications by the scalars $\pm1$. These two isometries can be distinguished canonically by their determinant, which is either $1$ or $-1$.

Defining $\Gamma$ as the subgroup $\widetilde\Gamma$ consisting of isometries of determinant $1$ gives a free and transitive action on the non-ideal points of $\mathcal{L}$.

If $Q(L)\ne 0$, the quadratic space $L^\perp$ is non-degenerate, and $\widetilde\Gamma$ is determined by the quadratic form on it. Otherwise $L$ can be embedded in a generalized basis and the quadratic form takes the form $[+1,+1,-1]$ up to scalar. Since all isotropic vectors are in the same orbit of the complete orthogonal group, $\widetilde\Gamma$ is determined.
\end{proof}

\begin{remark}
A similar theorem can be proven for perfect fields with $\charop\mathbb{K}=2$, but it necessitates the introduction of more technical concepts. In particular, we can still introduce an index $2$ subgroup, even though it can not be identified by the determinant.
\end{remark}

\begin{definition}
Let us define the {\bf oriented distance} between two points $p_1$ and $p_2$ of the line $\mathcal{L}$ as the unique element $\gamma\in\Gamma$ that sends $p_1$ to $p_2$. This distance is additive, and swapping $p_1$ with $p_2$ gives $\gamma^{-1}$. The {\bf distance} is then an element of the set $\Gamma/_{\gamma\sim\gamma^{-1}}$. This distance is preserved by all isometries of the conformal geometry.
\end{definition}

$\Gamma$ can then be determined from the quadratic form on $\mathcal{L}$ and the norm of $L$. Note that this also means that the distance between two points is not always in the same group, as it depends on the type of line that connects them.

\begin{example}
Consider the real case. Here, $\Gamma$ is one of $SO(2)\cong \mathbb{T}^1$, $SO(1,1)\cong\mathbb{R}^\times$ or $\mathbb{R}^+$. In the first case, distance is measured by an angle, as is seen in the elliptic case, while in the last one, it is measured additively, as in the Euclidean case. The elements of $SO(1,1)$, used in hyperbolic geometry, are usually written as matrices containing $\cosh(d)$ and $\sinh(d)$, however by replacing the matrix with $\cosh(d)+\sinh(d)=e^{2d}$, we may see that the usual distance may be considered by taking the logarithm $\log\colon\mathbb{R}^\times\to\mathbb{R}^+$.
\end{example}

\begin{example}
Consider a finite field $\mathbb{F}_q$ with $\charop\mathbb{F}_q=p\ne 2$. If ${\bf e}$ is a non-square, the quadratic polynomial $x^2-{\bf e}y^2=0$ has no non-trivial solution, hence there are no isotropic vector for the quadratic form $[1,-{\bf e}]$. The other non-isomorphic form is $[1,-1]$, where the directions $(1,1)$ and $(1,-1)$ are isotropic. These may be considered, by analogy with the real cases, the ellipse and hyperbola over finite fields.
\end{example}

Angles can be defined in a dual way, by exchanging $P$ with $L$.

\section{Cayley-Klein geometries}

Now let us consider the case when the dimension of the geometry is $2$ and the characteristic is not $2$. Recall that the norm of a point in $\mathbb{P}V$ is well defined up to multiplying by a square. To classify the possible $2$-geometries, consider the signature of the quadratic form $Q$ and the norm of $P$ and $L$.

\begin{lemma}
If $\charop\mathbb{K}\ne 2$, the quadratic form of a non-degenerate geometry of dimension $2$ is isomorphic to $[+1,+1,+1,-1,-1]$ up to scalar.
\end{lemma}

\begin{proof}
If the geometry is non-degenerate, then it has a symplectic subspace of dimension $4$. Hence the quadratic form is $[+1,+1,-1,-1]\oplus[A]$ for some $A\ne 0$. Multiplying the form by $1/A$, and noting that multiplying a symplectic space by a scalar is still symplectic, we get $[+1,+1,-1,-1]\oplus[+1]$.
\end{proof}

Over the reals and finite fields $\mathbb{F}_q$ ($2\nmid q$), multiplying the quadratic form $[+1,+1,+1,-1,-1]$ by a non-square gives a non-isomorphic quadratic form. Hence the pair $(Q(L),Q(P))$, given up to squares, uniquely determines the geometry: for the reals it may be $0$, $+1$ or $-1$, for a finite field and quadratic non-residue ${\bf e}$, it may be $0$, $1$ or ${\bf e}$. This gives us $3\times 3$ possibilities.

Cayley-Klein geometries are classified by the groups of translations along a line, and rotations around a point. As seen earlier, the group of translations along a line $\mathcal{L}$ is determined by the quadratic form restricted to $\mathcal{L}$ and the norm of $L$. A simple result is given here without proof:

\begin{lemma}
Let the dimension of the geometry be $2$, and let $\mathcal{L}$ be a virtual line that is not quasi-ideal, i.e. the quadratic form is non-degenerate restricted to it. If $P$ is anisotropic, $\mathcal{L}$ is given as the dual of some vector $\mathcal{L}^*\in P^\perp$, and the quadratic form on $\mathcal{L}$ only depends on $Q(\mathcal{L}^*)$. If $P$ is isotropic, the isomorphism class of the quadratic form on $\mathcal{L}$ is independent of the choice of $\mathcal{L}$.
\end{lemma}

However, when considering only actual lines, i.e. those that are of the form $\langle P,l\rangle^\perp$ for some cycle $l$, these are the dual of $l^P\in P^\perp$ where $l^P$ is the projection of $l$ to $P^\perp$. The quadratic form on the line is then given by $Q(l^P)=-\frac{1}{4}\frac{B(P,l)^2}{Q(P)}$, which in this case is entirely determined by $Q(P)$ up to a square. One of the central results in this article is that in dimension $2$, the translation and rotation groups are determined by the norms of $P$ and $L$, independently of the choice of the line and angle.

\begin{theorem}
Given a two dimensional geometry, there is a unique group of translations $\Gamma_t$ and rotations $\Gamma_r$, depending only on the norm of $L$ and $P$, respectively, such that the distance between any two points and the angle between any two oriented lines is given as an element in these groups up to inverse.
\end{theorem}

\begin{proof}
We will calculate the group $\Gamma_t$ as an orthogonal group, and a similar calculation can be done for $\Gamma_r$.

We may assume that the quadratic form is of the form $[+1,+1,+1,-1,-1]$, which holds up to a scalar. Choose a non-ideal line $l$, i.e. $Q(l)=0$, $B(L,l)=0$, $B(P,l)\ne 0$. The space $\langle P,l\rangle$ is symplectic (hence isomorphic to $[+1,-1]$), and the quadratic form restricted to $\mathcal{L}:=\langle P,l\rangle^\perp$ is of the form $[+1,+1,-1]$.

Now consider that a symplectic space of the form $[+1,-1]$ can be reparametrized as $[+\lambda,-\lambda]$ for some scalar $\lambda$. Assuming that $Q(L)\ne 0$, $[+1,+1,-1]$ is isomorphic to $[+1,Q(L),-Q(L)]$, and $\Gamma_t$ is the special orthogonal group preserving the quadratic form $[+1,-Q(L)]$.

If however, $Q(L)=0$, we may choose a generalized basis $u$, $v$ such that $B(L,u)=0$, $B(L,v)=1$ and $Q(v)=0$. Any isometry that fixes $L$ preserves the subspace $L^\perp$, hence the image of $u$ is in $\langle L,u\rangle$. Since the norm of $u$ must be preserved and $Q(\varepsilon u+\tau L)=\varepsilon^2 Q(u)$, the image of $u$ must be of the form $\pm u+\tau L$. A simple calculation tells us that such an isometry on $\langle L,u\rangle$ extends uniquely to $v$ (by sending it to $v\mp\frac{\tau}{Qu}u-\frac{\tau^2}{Qu}L$), and that the determinant is $1$ if and only if the image of $u$ is chosen as $u+\tau L$. If we denote this map by $T_\tau$, then $T_{\tau_1}\circ T_{\tau_2}=T_{\tau_1+\tau_2}$, hence $\Gamma_t\cong\mathbb{K}^+$.
\end{proof}

\begin{example}
In the real case, the norm of $L$ may be $1$, $0$ or $-1$, giving $\Gamma_t=SO(2)=\mathbb{T}^1$, $\Gamma_t=\mathbb{R}^+$, $\Gamma_t=SO(1,1)=\mathbb{R}^\times$, respectively. When $\Gamma_t=\mathbb{T}^1$ or $\mathbb{R}^\times$, we can create covers by $\mathbb{R}^+$.

For a finite field and quadratic non-residue ${\bf e}$, there are two quadratic forms in $2$ dimensions, and $[1,-{\bf e}]$ is the non-symplectic case. Therefore choosing $L$ to be ${\bf e}$, $0$ and $1$ gives cases analoguous to the real field.
\end{example}

Over the reals, by choosing whether distances are measured along the groups $SO(2)$ if $Q(L)=-1$, $SO(1,1)$ if $Q(L)=1$ and $\mathbb{R}$ if $Q(L)=0$, and analoguously for angles and $P$, the $3\times 3$ possibilities represent the $9$ Cayley-Klein geometries, as seen in \Str{}:
\begin{center}
\begin{tabular}{|c|c|c|c|}
\hline
\diagbox{$Q(P)$}{$Q(L)$}&	$-1$&	$0$&	$1$\\
\hline
$-1$&		elliptic&		parabolic&		hyperbolic\\
\hline
$0$&		dual parabolic&		Laguerre/Galilei&	dual Minkowski\\
\hline
$1$&		dual hyperbolic&		Minkowski&		anti-de Sitter\\
\hline
\end{tabular}
\end{center}

Interestingly, we get a very similar classification for finite fields $\mathbb{F}_q$ of characteristic $p\ne 2$, but the $-1$ has to be replaced with ${\bf e}$.

For the complex field (and in fact any quadratically closed field), the signs of the norms are not relevant anymore, and the elliptic, hyperbolic, dual hyperbolic and anti-de Sitter cases are isometric geometries, and similarly for the parabolic and Minkowski geometries, and also their duals.

Finally, it is worth noting that some of these geometries are cycle equivalent, meaning that they are geometries over the same set of points and unoriented cycles, but with different sets of orientable cycles.

\begin{theorem}
Over the reals, only the dual hyperbolic space and the anti-de Sitter spaces are cycle equivalent, while the Minkowski space has two cycle equivalent, non-isomorphic models.

Over the finite fields of odd characteristic, any pair of geometries with $Q(P)\ne 0$ are cycle equivalent if $Q(L)\ne 0$, and if $Q(L)=0$, it has two cycle equivalent, non-isomorphic models.
\end{theorem}

\begin{proof}
Two spaces are cycle equivalent if $Q|_{P^\perp}$ is equal for the two. Hence to get from one space to another, it suffices to extend $Q|_{P^\perp}$ differently to the whole space. If $Q(P)=0$, this can be done in a single manner up to isometry, hence it does not result in a different geometry. Therefore we may assume $Q(P)\ne 0$ and that $\langle P,P^\perp\rangle$ generate the entire space.
The quadratic form is isomorphic to $[+1,+1,+1,-1,-1]$ up to scalar. By denoting $p_1:=Q(P)$, $Q|_{P^\perp}$ is isomorphic to $[+1,+1,-p_1,-1]$, and we may extend it with a new norm $Q'(P):=p_2$, giving $[+1,+1,p_2,-p_1,-1]$. In order to get a non-degenerate geometry, this must be isomorphic to the original up to a scalar.

Over the reals, the signature must be either $(3,2)$ or $(2,3)$, hence $p_1=1$, $p_2=-1$ up to a square, and the scalar that it must be multiplied with is $-1$ in order to get an isomorphic quadratic form. Hence in the new geometry, the quadratic form is $-Q'$, giving $-Q'(P)=-p_2=1=Q(P)$ and $-Q'(L)=-Q(L)$, preserving the sign of $Q(P)$ but changing that of $Q(L)$.

Over the finite fields of odd characteristic, the determinant of the quadratic form determines the isomorphism class. Since $\det[+1,+1,+1,-1,-1]=1$ and $\det[+1,+1,p_2,-p_1,-1]=p_1p_2$, the two forms are isomorphic if multiplied with the scalar $p_1p_2$. Take a non-square ${\bf e}$ from the field, and considering $p_1\ne p_2$, we may assume $p_1p_2={\bf e}$ up to a square. Hence the new quadratic form is ${\bf e}Q'$, giving ${\bf e}Q'(P)={\bf e}p_2$, which is equal to $p_1$ up to a square, and ${\bf e}Q'(L)={\bf e}Q(L)$, preserving $Q(P)$ but exchanging $Q(L)$ with ${\bf e}Q(L)$.

Finally note that there is no non-square scalar that preserves the form $[+1,+1,+1,-1,-1]$, hence if $p_1=p_2$, we do not get non-isomorphic models of the geometry.
\end{proof}

\section{Further research}

After the introduction of distance and angle, we can create a generalized form of trigonometry over all spaces. The details of this, however, would need more technical definitions, and should probably form a separate work.

All the definitions may be applied to the case of characteristic $2$ as well, however as usual, there are certain obstacles to developing an analogous theory for distances. The introduction of the necessary technical tools goes over the scope of this article, and will be treated separately.

\section*{Acknowledgements}

Acknowledgements are due to Andr\'as Hrask\'o who introduced me to Lie sphere geometry, and provided me with important references.

Finally, I would like to thank Professor Emil Moln\'ar for his comments and for providing some essential references,
and the Budapest University of Technology and Economics for their support.

\section*{Conflict of interest statement}

On behalf of all authors, the corresponding author states that there is no conflict of interest.

\end{document}